\documentclass[12pt]{article}
\usepackage{amssymb}
\usepackage{amsthm}
\usepackage{amsmath}
\usepackage{graphicx}
\usepackage{enumerate}

\newtheorem{theorem}{Theorem}
\newtheorem{definition}[theorem]{Definition}
\newtheorem{lemma}[theorem]{Lemma}
\newtheorem{proposition}[theorem]{Proposition}

\newtheorem{remark}[theorem]{Remark}
\newtheorem{example}[theorem]{Example}
\newtheorem{corollary}[theorem]{Corollary}
\title{c-ideals in complemented posets}
\author{Ivan~Chajda, Miroslav~Kola\v r\'ik and Helmut~L\"anger}
\date{}
\begin{document}

\footnotetext{Support of the research of the first and third author by the Austrian Science Fund (FWF), project I~4579-N, and the Czech Science Foundation (GA\v CR), project 20-09869L, entitled ``The many facets of orthomodularity'', is gratefully acknowledged.}

\maketitle

\begin{abstract}
In their recent paper on posets with a pseudocomplementation denoted by $*$ the first and the third author introduced the concept of a $*$-ideal. This concept is in fact an extension of a similar concept introduced in distributive pseudocomplemented lattices and semilattices by several authors, see References. Now we apply this concept of a c-ideal (dually, c-filter) to complemented posets where the complementation need neither be antitone nor an involution, but still satisfies some weak conditions. We show when an ideal or filter in such a poset is a c-ideal or c-filter, respectively, and we prove basic properties of them. Finally, we prove so-called Separation Theorems for c-ideals. The text is illustrated by several examples.
\end{abstract}



The concept of a $\delta$-ideal was introduced recently in pseudocomplemented distributive lattices and semilattices in \cite{NN}, \cite R and \cite{TCB}. Later on, it was extended to pseudocomplemented posets under the name $*$-ideal in \cite{CL} where $*$ means the pseudocomplementation on the poset in question. The authors used also some results taken from their previous paper \cite{CL21}.

It turns out that in complemented posets the aforementioned concepts of ideals and filters play also an important role and this fact motivated us to extend our study to complemented posets where the complementation need not be an antitone involution in all cases. Several such examples are included in the paper. Hence, we obtain the concept of a c-ideal. Our goals are to present several basic properties of c-ideals and prove so-called Separation Theorems showing that under certain assumptions for every ideal $I$ and certain filters $F$ of a complemented poset such that $I\cap F=\emptyset$ there exists a c-ideal $J$ including $I$ with $J\cap F=\emptyset$.

In what follows we collect the concepts used throughout the paper. Some of them are familiarly known and can be found e.g.\ in \cite B.

Let $(P,\leq)$ be a poset and $a,b\in P$ and $A,B\subseteq P$. We define
\begin{align*}
L(A) & :=\{x\in P\mid x\leq y\text{ for all }y\in A\}, \\
U(A) & :=\{x\in P\mid y\leq x\text{ for all }y\in A\},
\end{align*}
the so-called {\em lower cone} and {\em upper cone} of $A$, respectively. Instead of $L(\{a\})$, $L(\{a,b\})$, $L(A\cup\{a\})$, $L(A\cup B)$ and $L\big(U(A)\big)$ we simply write $L(a)$, $L(a,b)$, $L(A,a)$, $L(A,B)$ and $LU(A)$, respectively. Analogously, we proceed in similar cases.

Consider a {\em bounded} poset $(P,\leq,0,1)$, i.e.\ a poset with bottom element $0$ and top element $1$. A unary operation ${}'$ on $P$ is called a {\em complementation} if for every $x\in P$ there exist $x\vee x'$ and $x\wedge x'$ and, moreover, $x\vee x'=1$ and $x\wedge x'=0.$ If ${}'$ is a complementation on $(P,\leq,0,1)$ then $(P,\leq,{}',0,1)$ will be called a {\em complemented poset}. Clearly, $0'=0\vee0'=1$ and $1'=1\wedge1'=0$.

A unary operation ${}'$ on a poset $(P,\leq)$ is called an {\em involution} if it satisfies the identity $x''\approx x$ and it is called {\em antitone} if $x,y\in P$ and $x\leq y$ together imply $y'\leq x'$.

Let us note that if ${}'$ is an antitone involution on a given poset $(P,\leq)$ then ${}'$ satisfies the identity $x'''\approx x'$ and we can use {\em De Morgan's laws}, i.e.\ $\big(L(x,y)\big)'=U(x',y')$ and $\big(U(x,y)\big)'=L(x',y')$ for all $x,y\in P$.

Let $\mathbf P=(P,\leq,{}',0,1)$ be a complemented poset. For a subset $A$ of $P$ we define
\begin{align*}
 A' & :=\{x'\mid x\in A\}, \\
A_0 & :=\{x\in P\mid x'\in A\}.
\end{align*}
It is an easy observation that $A\subseteq B$ implies $A_0\subseteq B_0$. An element $a$ of $P$ is called {\em Boolean} if $a''=a$. For non-empty subsets $I$ and $F$ of $P$ we define

$I$ is called an {\em ideal} of $\mathbf P$ if $L(x)\subseteq I$ and $U(x,y)\cap I\neq\emptyset$ for all $x,y\in I$, \\
$F$ is called a {\em filter} of $\mathbf P$ if $U(x)\subseteq F$ and $L(x,y)\cap F\neq\emptyset$ for all $x,y\in F$.

\begin{lemma}
Let $\mathbf P=(P,\leq,{}',0,1)$ be a complemented poset with antitone complementation satisfying $x\leq x''$ for all $x\in P$, let $a\in P$ and assume that every ideal of $\mathbf P$ containing $a$ contains $a''$. Then $a$ is a Boolean element of $\mathbf P$.
\end{lemma}

\begin{proof}
Everyone of the following statements implies the next one: $a\in L(a)$ and $L(a)$ is an ideal of $\mathbf P$, $a''\in L(a)$, $a\leq a''\leq a$, $a''=a$, $a$ is a Boolean element of $\mathbf P$.
\end{proof}

For an ideal $I$ of $\mathbf P$ we say that

$I$ is called {\em proper} if $I\neq P$, \\
$I$ is called {\em maximal} if $I$ is a maximal proper ideal of $\mathbf P$, \\
$I$ is called a {\em prime ideal} if $I\neq P$, and $\{x,y\}\cap I\neq\emptyset$ for all $x,y\in P$ with $L(x,y)\subseteq I$.

For an filter $F$ of $\mathbf P$ we define

$F$ is called {\em proper} if $F\neq P$, \\
$F$ is called an {\em ultrafilter} if $F$ is a maximal proper filter of $\mathbf P$, \\
$I$ is called a {\em prime filter} if $F\neq P$, and $\{x,y\}\cap F\neq\emptyset$ for all $x,y\in P$ with $L(x,y)\subseteq F$.

Now we define our main concepts.

An ideal $I$ of $\mathbf P$ is called a {\em c-ideal} if there exists some filter $F$ of $\mathbf P$ with $F_0=I$, \\
A filter $F$ of $\mathbf P$ is called a {\em c-filter} if there exists some ideal $I$ of $\mathbf P$ with $I_0=F$.

Namely, these concepts enable to separate ideals from filters in complemented posets as expressed in the so-called {\em Separation Theorems} Theorems~\ref{th3} and \ref{th4}).

It is evident that the concept of an ideal and a filter are dual to each other.

Clearly, $0,1$ are Boolean elements of a complemented poset $\mathbf P$. Further, $\{0\}$ and $P$ are ideals of $\mathbf P$, and $0\in I$ for all ideals $I$ of $\mathbf P$. Moreover, an ideal $I$ of $\mathbf P$ is proper if and only if $1\notin I$. Dual statements hold for filters.

The ideals of the form $L(a)$ with $a\in P$ are called {\em principal ideals}, and the filters of the form $U(a)$ with $a\in P$ {\em principal filters}.

Let us repeat the following useful results from \cite{CL}.

\begin{lemma}
{\rm(\cite{CL})} Let $\mathbf P=(P,\leq)$ be a poset and $I$ an ideal of $\mathbf P$. Then the following are equivalent:
\begin{enumerate}[{\rm(i)}]
\item $I$ is a prime ideal of $\mathbf P$.
\item $P\setminus I$ is a prime filter of $\mathbf P$.
\item $P\setminus I$ is a filter of $\mathbf P$.
\end{enumerate}
Moreover, the mapping $I\mapsto P\setminus I$ is a bijection from the set of all prime ideals of $\mathbf P$ to the set of all prime filters of $\mathbf P$.
\end{lemma}

Recall that a poset $(P,\leq)$ satisfies the {\em Ascending Chain Condition} if it has no infinite ascending chains. The {\em Descending Chain Condition} is defined dually.

\begin{lemma}\label{lem3}
{\rm(\cite{CL})} Let $\mathbf P=(P,\leq)$ be a poset. Then the following hold
\begin{enumerate}[{\rm(i)}]
\item Every ideal of $\mathbf P$ is principal if and only if $\mathbf P$ satisfies the Ascending Chain Condition,
\item every filter of $\mathbf P$ is principal if and only if $\mathbf P$ satisfies the Descending Chain Condition.
\end{enumerate}
\end{lemma}

Note that if $P$ is finite then $(P,\leq)$ satisfies the Ascending Chain Condition as well as the Descending Chain Condition.

\begin{lemma}\label{lem1}
Let $\mathbf P=(P,\leq,{}',0,1)$ be a complemented poset, $a\in P$ and $I$ a proper ideal of $\mathbf P$. Then either $a\notin I$ or $a'\notin I$.
\end{lemma}

\begin{proof}
If $a,a'\in I$ then $\{1\}\cap I=U(a,a')\cap I\neq\emptyset$, i.e.\ $1\in I$ which implies $I=P$ contradicting the assumption of $I$ being a proper ideal of $\mathbf P$.
\end{proof}

In the following we demonstrate the role of $I_0$.

\begin{proposition}\label{prop1}
Let $\mathbf P=(P,\leq,{}',0,1)$ be a complemented poset and $I$ an ideal and $F$ a filter of $\mathbf P$. Then the following are equivalent:
\begin{enumerate}[{\rm(i)}]
\item $I$ is proper,
\item $I_0\neq P$,
\item $I\cap I_0=\emptyset$.
\end{enumerate}
Moreover, the following are equivalent:
\begin{enumerate}
\item[{\rm(iv)}] $F$ is proper,
\item[{\rm(v)}] $F_0\neq P$,
\item[{\rm(vi)}] $F\cap F_0=\emptyset$.
\end{enumerate}
\end{proposition}

\begin{proof}
$\text{}$ \\
(i) $\Rightarrow$ (ii): \\
$I_0=P$ would imply $0\in I_0$ and hence $1=0'\in I$, i.e.\ $I=P$, a contradiction. \\
(ii) $\Rightarrow$ (iii): \\
Suppose, $I\cap I_0\neq\emptyset$. Then there exists some $a\in I\cap I_0$. Hence $a,a'\in I$. Since $I$ is an ideal of $\mathbf P$ we conclude $\{1\}\cap I=U(b,b')\cap I\neq\emptyset$, i.e.\ $1\in I$ and therefore $I=P$ whence $I_0=P$, a contradiction. \\
(iii) $\Rightarrow$ (i): \\
$I=P$ would imply $I\cap I_0=P\cap P=P\neq\emptyset$, a contradiction. \\
The proof for filters follows by duality.
\end{proof}

\begin{lemma}\label{lem4}
Let $\mathbf P=(P,\leq,{}',0,1)$ be a complemented poset and $I$ a c-ideal and $F$ a c-filter of $\mathbf P$. Then the following hold:
\begin{enumerate}[{\rm(i)}]
\item Assume $x'\leq x'''$ for all $x\in P$. Then $I''\subseteq I$.
\item Assume $x'''\leq x'$ for all $x\in P$. Then $F''\subseteq F$.
\end{enumerate}
\end{lemma}

\begin{proof}
\
\begin{enumerate}[(i)]
\item Since $I$ is a c-ideal of $\mathbf P$ there exists some filter $F$ of $\mathbf P$ with $F_0=I$ and everyone of the following statements implies the next one: $x\in I$, $x\in F_0$, $x'\in F$, $x'''\in F$, $x''\in F_0$, $x''\in I$.
\item follows by duality.
\end{enumerate}
\end{proof}

\begin{example}\label{ex1}
Consider the complemented poset $\mathbf P=(P,\leq,{}',0,1)$ shown in Fig.~1 and the table for its complementation:

\vspace*{-4mm}

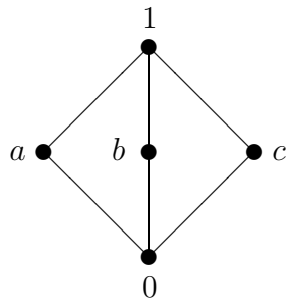
\begin{center}
\setlength{\unitlength}{7mm}
\begin{picture}(12,8)
\put(6,2){\circle*{.3}}
\put(6,4){\circle*{.3}}
\put(6,6){\circle*{.3}}
\put(4,4){\circle*{.3}}
\put(8,4){\circle*{.3}}
\put(6,2){\line(-1,1)2}
\put(6,2){\line(1,1)2}
\put(6,2){\line(0,1)4}
\put(4,4){\line(1,1)2}
\put(8,4){\line(-1,1)2}
\put(5.875,1.25){$0$}
\put(3.35,3.85){$a$}
\put(5.3,3.85){$b$}
\put(8.35,3.85){$c$}
\put(5.875,6.35){$1$}
\put(5.3,.3){{\rm Fig.~1}}
\put(4.2,-.5){bounded poset}
\end{picture}
\raisebox{15ex}{
\begin{tabular}{c|ccccc}
  $x$   & $0$ & $a$ & $b$ & $c$ & $1$ \\
  \hline
  $x'$  & $1$ & $b$ & $c$ & $b$ & $0$ \\
  \hline
  $x''$ & $0$ & $c$ & $b$ & $c$ & $1$ \\
  \hline
  $x'''$ & $1$ & $b$ & $c$ & $b$ & $0$ \\
\end{tabular}}
\end{center}

\vspace*{5mm}

We have

Boolean elements: $0$, $b$, $c$, $1$, \\
maximal ideals: $L(a)$, $L(b)$, $L(c)$ \\
ultrafilters: $U(a)$, $U(b)$, $U(c)$, \\
$\mathbf P$ has neither prime ideals nor prime filters, \\
c-ideals: $L(0)$, $L(b)$, $L(1)$, \\
c-filters: $U(0)$, $U(b)$, $U(1)$.

The complementation defined by the above table is antitone and satisfies the identity $x'''\approx x'$, but it is not an involution. We have
\begin{align*}
L(0)'' & =L(0), \\
L(a)'' & =L(c)\not\subseteq L(a), \\
L(b)'' & =L(b), \\
L(c)'' & =L(c), \\
L(1)'' & =\{0,b,c,1\}\subseteq L(1).
\end{align*}
\end{example}

\begin{lemma}\label{lem5}
Let $(P,\leq,{}',0,1)$ be a complemented poset satisfying $x'''\approx x'$ and let $a\in P$ and $A\subseteq P$. Then $a\in A_0$ if and only if $a''\in A_0$.
\end{lemma}

\begin{proof}
The following are equivalent: $a\in A_0$, $a'\in A$, $a'''\in A$, $a''\in A_0$.
\end{proof}

If we suppose that the complementation is antitone then we can formulate easy assumptions ensuring that every ideal is a c-ideal and every filter is a c-filter.

\begin{theorem}\label{th1}
Let $\mathbf P=(P,\leq,{}',0,1)$ be a complemented poset with antitone complementation and $F$ a filter and $I$ an ideal of $\mathbf P$. Then the following hold:
\begin{enumerate}[{\rm(i)}]
\item Assume $x\leq x''$ for all $x\in P$. Then $F_0$ is a c-ideal of $\mathbf P$.
\item Assume $x''\leq x$ for all $x\in P$. Then $I_0$ is a c-filter of $\mathbf P$.
\end{enumerate}
\end{theorem}

\begin{proof}
\
\begin{enumerate}[(i)]
\item Let $a,b\in F_0$. Since $0'=1\in F$, we have $0\in F_0$ and hence $F_0\neq\emptyset$. Because of $a',b'\in F$ and $F$ is a filter of $\mathbf P$ there exists some $c\in L(a',b')\cap F$. Since $c\leq c''$ we have $c''\in F$, i.e.\ $c'\in F_0$. Together $c'\in U(a'',b'')\cap F_0\subseteq U(a,b)\cap F_0$ proving $U(a,b)\cap F_0\neq\emptyset$. If $d\in P$, $e\in F_0$ and $d\leq e$ then $e'\in F$ and $e'\leq d'$ and hence $d'\in F$, i.e.\ $d\in F_0$. Altogether, $F_0$ is an ideal and $F$ a filter of $\mathbf P$ and hence $F_0$ is a c-ideal of $\mathbf P$.
\item follows by duality.
\end{enumerate}
\end{proof}

If we assume that the complementation is an antitone involution, which is a rather strong assumption, we can state the following result.

\begin{corollary}\label{cor2}
Let $\mathbf P=(P,\leq,{}',0,1)$ be a complemented poset with an antitone involution and $I$ an ideal and $F$ a filter of $\mathbf P$. Then
\begin{enumerate}[{\rm(i)}]
\item $I_0$ is a filter of $\mathbf P$, $I=(I_0)_0$ and hence $I$ is a c-ideal of $\mathbf P$,
\item $F_0$ is an ideal of $\mathbf P$, $F=(F_0)_0$ and hence $F$ is a c-filter of $\mathbf P$.
\end{enumerate}
\end{corollary}

\begin{proof}
\
\begin{enumerate}[(i)]
\item From (ii) of Theorem~\ref{th1} we obtain that $I_0$ is a filter of $\mathbf P$. Moreover, for $x\in P$ the following are equivalent: $x\in(I_0)_0$, $x'\in I_0$, $x''\in I$, $x\in I$. This shows $(I_0)_0=I$. By (i) of Theorem~\ref{th1} we conclude that $I$ is a c-ideal of $\mathbf P$.
\item follows by duality.
\end{enumerate}
\end{proof}

\begin{remark}\label{rem1}
If $(P,\leq,{}',0,1)$ is a complemented poset whose complementation is an antitone involution and $a\in P$ then $L(a)_0=U(a')$ since the following are equivalent: $x\in L(a)_0$, $x'\in L(a)$, $x'\leq a$, $a'\leq x$, $x\in U(a')$. According to {\rm(ii)} of Corollary~\ref{cor2} we have $L(a)=U(a')_0$.
\end{remark}

Now we formulate a condition which will be helpful for the Separation Theorems.

\begin{definition}\label{def1}
A subset $A$ of a complemented poset $(P,\leq,{}',0,1)$ satisfies the {\em c-condition} if for every $x\in P$ the set $A$ contains exactly one of $x$ and $x'$.
\end{definition}

Recall from Lemma~\ref{lem1} that for a proper ideal $I$ of $\mathbf P$ and for arbitrary $x\in P$ the situation $x,x'\in I$ is impossible.

The following lemma shows that ideals satisfying c-condition can be found among prime ideals.

\begin{lemma}\label{lem2}
Let $\mathbf P=(P,\leq,{}',0,1)$ be a complemented poset and $I$ a prime ideal and $F$ a prime filter of $\mathbf P$. Then $I$ and $F$ satisfy the c-condition.
\end{lemma}

\begin{proof}
If $a\in P$ then $L(a,a')=\{0\}\subseteq I$ and hence $\{a,a'\}\cap I\neq\emptyset$. The rest follows by duality.
\end{proof}

Recall from {\rm\cite{LR}} that a {\em poset} $(P,\leq)$ is called {\em distributive} if it satisfies one of the following equivalent conditions:

$L\big(U(x,y),z\big)=LU\big(L(x,z),L(y,z)\big)$ for all $x,y,z\in P$, \\
$U\big(L(x,y),z\big)=UL\big(U(x,z),U(y,z)\big)$ for all $x,y,z\in P$.

The next result illuminates the role of distributivity of the poset $\mathbf P$ for the c-condition both for ideals and filters of $\mathbf P$.

\begin{theorem}\label{th5}
Let $\mathbf P=(P,\leq,{}',0,1)$ be a complemented poset and $I$ an ideal and $F$ a filter of $\mathbf P$. Consider the following statements:
\begin{enumerate}[{\rm(i)}]
\item $I$ satisfies the c-condition,
\item $I$ is a maximal ideal of $\mathbf P$,
\item $(P,\leq)$ is distributive,
\item $\bigcup\{LU(a,i)\mid i\in I\}$ is an ideal of $\mathbf P$ for all $a\in P\setminus I$,
\item $F$ satisfies the c-condition,
\item $F$ is an ultrafilter of $\mathbf P$,
\item $\bigcup\{UL(a,f)\mid f\in F\}$ is a filter of $\mathbf P$ for all $a\in P\setminus F$.
\end{enumerate}
Then \\
{\rm(i)} implies {\rm(ii)}, \\
{\rm(ii)}, {\rm(iii)} and {\rm(iv)} imply {\rm(i)}, \\
{\rm(v)} implies {\rm(vi)}, \\
{\rm(iii)}, {\rm(vi)} and {\rm(vii)} imply {\rm(v)}.
\end{theorem}

\begin{proof}
$\text{}$ \\
(i) $\Rightarrow$ (ii): \\
Since $I$ satisfies the c-condition, $I\neq P$. Let $J$ be an ideal of $\mathbf P$ strictly including $I$. Then there exists some $a\in J\setminus I$. Because of (i) we conclude $a'\in I$ which implies $a'\in J$. Since $J$ is an ideal of $\mathbf P$ we have $\{1\}\cap J=U(a,a')\cap J\neq\emptyset$ and hence $1\in J$ which implies $J=P$. \\
$\big($(ii), (iii) and (iv)$\big)$ $\Rightarrow$ (i): \\
Let $a\in P\setminus I$. Because of (iv), $K:=\bigcup\{LU(a,i)\mid i\in I\}$ is an ideal of $\mathbf P$ including $I\cup\{a\}$ and hence strictly including $I$. Since $I$ is a maximal ideal of $\mathbf P$ we conclude $K=P$. Hence $1\in K$ and therefore there exists some $i\in I$ with $1\in LU(a,i)$. This means $U(a,i)=\{1\}$. Using (iii) we have
\begin{align*}
i\in U(i) & =U(0,i)=U\big(L(a,a'),i\big)=UL\big(U(a,i),U(a',i)\big)=UL\big(1,U(a',i)\big)= \\
          & =ULU(a',i)=U(a',i)\subseteq U(a'),
\end{align*}
i.e.\ $a'\leq i$. Since $i\in I$ and $I$ is an ideal of $\mathbf P$ we conclude $a'\in I$. The rest follows by duality.
\end{proof}

The following lemma shows that condition (iv) of Theorem~\ref{th5} is satisfied automatically if the poset $\mathbf P$ in question is a join-semilattice. Dually, condition (vii) of Theorem~\ref{th5} holds if $\mathbf P$ is a meet-semilattice.

\begin{lemma}
Let $\mathbf P=(P,\leq)$ be a join-semilattice, $I$ an ideal of $\mathbf P$ and $a\in P$. Then $\bigcup\{LU(a,i)\mid i\in I\}$ is an ideal of $\mathbf P$.
\end{lemma}

\begin{proof}
If $b,c\in I$ then there exists some $d\in U(b,c)\cap I$ whence $b\vee c\leq d\in I$ and hence $b\vee c\in I$. Put $J:=\bigcup\{LU(a,i)\mid i\in I\}$. Then
\[
J=\bigcup\{LU(a\vee i)\mid i\in I\}=\bigcup\{L(a\vee i)\mid i\in I\}.
\]
If $b,c\in J$ then there exist $j,k\in I$ with $b\leq a\vee j$ and $c\leq a\vee k$ and hence
\[
b\vee c\in U(b,c)\cap L\big(a\vee(j\vee k)\big)\subseteq U(b,c)\cap J.
\]
Since $J$ is downward closed it is an ideal of $\mathbf P$.
\end{proof}

Now we are ready to prove our first Separation Theorem.

\begin{theorem}\label{th3}
{\rm(}{\bf First Separation Theorem}{\rm)} Let $\mathbf P=(P,\leq,{}',0,1)$ be a complemented poset with antitone complementation satisfying $x\leq x''$ for each $x\in P$ and let $I$ be an ideal and $F$ a filter of $\mathbf P$ satisfying the c-condition and $I\cap F=\emptyset$. Then there exists some c-ideal $J$ of $\mathbf P$ with $I\subseteq J$ and $J\cap F=\emptyset$.
\end{theorem}

\begin{proof}
By Theorem~\ref{th1}, $F_0$ is a c-ideal of $\mathbf P$. Since $I\neq\emptyset$ and $I\cap F=\emptyset$ we have $F\neq P$ which implies $F\cap F_0=\emptyset$ according to Proposition~\ref{prop1}. Since $F$ satisfies the c-condition, we have $F\cup F_0=P$. Thus $I\subseteq P\setminus F=F_0$. This shows that one may take $J:=F_0$.
\end{proof}

\begin{example}\label{ex2}
Consider the following bounded posets $\mathbf P_1$ and $\mathbf P_2$ which are not lattices:

\vspace*{-4mm}

\begin{center}
\setlength{\unitlength}{7mm}
\begin{picture}(10,12)
\put(6,2){\circle*{.3}}
\put(3,5){\circle*{.3}}
\put(5,4){\circle*{.3}}
\put(7,4){\circle*{.3}}
\put(9,5){\circle*{.3}}
\put(5,6){\circle*{.3}}
\put(7,6){\circle*{.3}}
\put(6,8){\circle*{.3}}
\put(6,10){\circle*{.3}}
\put(6,2){\line(-1,1)3}
\put(6,2){\line(-1,2)1}
\put(6,2){\line(1,2)1}
\put(6,2){\line(1,1)3}
\put(5,4){\line(0,1)2}
\put(7,4){\line(0,1)2}
\put(3,5){\line(3,5)3}
\put(9,5){\line(-3,5)3}
\put(6,8){\line(0,1)2}
\put(5,6){\line(1,2)1}
\put(7,6){\line(-1,2)1}
\put(5,4){\line(1,1)2}
\put(7,4){\line(-1,1)2}
\put(5.875,1.25){$0$}
\put(2.35,4.85){$g$}
\put(4.35,3.85){$a$}
\put(7.3,3.85){$b$}
\put(9.3,4.85){$f$}
\put(4.35,5.85){$c$}
\put(7.3,5.85){$d$}
\put(5.425,8.05){$e$}
\put(5.875,10.35){$1$}
\put(4.8,.3){{\rm Fig.~2 (a)}}
\put(9.4,-.5){bounded posets}
\end{picture}
\begin{picture}(10,12)
\put(6,2){\circle*{.3}}
\put(5,4){\circle*{.3}}
\put(7,4){\circle*{.3}}
\put(9,5){\circle*{.3}}
\put(5,6){\circle*{.3}}
\put(7,6){\circle*{.3}}
\put(6,8){\circle*{.3}}
\put(6,10){\circle*{.3}}
\put(6,2){\line(-1,2)1}
\put(6,2){\line(1,2)1}
\put(6,2){\line(1,1)3}
\put(5,4){\line(0,1)2}
\put(7,4){\line(0,1)2}
\put(9,5){\line(-3,5)3}
\put(6,8){\line(0,1)2}
\put(5,6){\line(1,2)1}
\put(7,6){\line(-1,2)1}
\put(5,4){\line(1,1)2}
\put(7,4){\line(-1,1)2}
\put(5.875,1.25){$0$}
\put(4.35,3.85){$a$}
\put(7.3,3.85){$b$}
\put(9.3,4.85){$f$}
\put(4.35,5.85){$c$}
\put(7.3,5.85){$d$}
\put(5.425,8.05){$e$}
\put(5.875,10.35){$1$}
\put(4.9,.3){{\rm Fig.~2 (b)}}
\end{picture}
\end{center}

\vspace*{5mm}

If the unary operation ${}'$ is defined by
\[
\begin{array}{c|ccccccccc}
 x  & 0 & a & b & c & d & e & f & g & 1 \\
  \hline
 x' & 1 & f & f & f & f & f & e & c & 0 \\
  \hline
x'' & 0 & e & e & e & e & e & f & f & 1 \\
\end{array}
\]
then $\mathbf P_1=(P_1,\leq,{}',0,1)$ and $\mathbf P_2=(P_2,\leq,{}',0,1)$ are complemented posets and the complementation is antitone, but not an involution. Moreover, it can be seen from the table that all elements $x$ except the element $g$ satisfy the inequality $x\leq x''$.
We have

$\mathbf P_1$:

Boolean elements: $0$, $e$, $f$, $1$, \\
maximal ideals: $L(e)$, $L(f)$, $L(g)$, \\
ultrafilters: $U(a)$, $U(b)$, $U(f)$, $U(g)$, \\
$\mathbf P_1$ has neither prime ideals nor prime filters, \\
c-ideals: $L(0)$, $L(e)$, $L(f)$, $L(1)$, \\
c-filters: $U(0)$, $U(g)$, $U(1)$, \\
$\mathbf P_1$ has no filter satisfying the c-condition.

$\mathbf P_2$:

Boolean elements: $0$, $e$, $f$, $1$, \\
maximal ideals: $L(e)$, $L(f)$, \\
ultrafilters: $U(a)$, $U(b)$, $U(f)$, \\
prime ideals: $L(e)$, \\
prime filters: $U(f)$, \\
c-ideals: $L(0)$, $L(e)$, $L(f)$, $L(1)$, \\
c-filters: $U(0)$, $U(f)$, $U(1)$, \\
filters satisfying the c-condition: $U(f)$.

$\mathbf P_1$ has no filter satisfying the c-condition. \\
The situation for $\mathbf P_2$ is different. Here $U(f)$ is the unique filter satisfying the c-condition. For every ideal $I\subseteq L(e)$ we have $I\cap U(f)=\emptyset$ and taking $J:=L(e)$, $J$ is a c-ideal of $\mathbf P_2$ with $I\subseteq J$ and $J\cap U(f)=\emptyset$.
\end{example}

Example~\ref{ex2} shows that the implication (ii) $\Rightarrow$ (i) in Theorem \ref{th5} does not hold in general.

The next result is in fact another version of the First Separation Theorem where we use the result from Lemma~\ref{lem2}.

\begin{corollary}\label{cor1}
Let $\mathbf P=(P,\leq,{}',0,1)$ be a complemented poset with antitone complementation satisfying $x\leq x''$ for each $x\in P$ and let $I$ be an ideal and $F$ a prime filter of $\mathbf P$ and assume $I\cap F=\emptyset$. Then there exists some c-ideal $J$ of $\mathbf P$ with $I\subseteq J$ and $J\cap F=\emptyset$.
\end{corollary}

\begin{proof}
From Lemma~\ref{lem2}  we conclude that $F$ satisfies the c-condition. Now apply Theorem~\ref{th3}.
\end{proof}

\begin{example}
Consider the complemented poset $\mathbf P$ visualized in Fig.~3:

\vspace*{-4mm}

\begin{center}
\setlength{\unitlength}{7mm}
\begin{picture}(8,8)
\put(4,1){\circle*{.3}}
\put(1,3){\circle*{.3}}
\put(3,3){\circle*{.3}}
\put(5,3){\circle*{.3}}
\put(7,3){\circle*{.3}}
\put(1,5){\circle*{.3}}
\put(3,5){\circle*{.3}}
\put(5,5){\circle*{.3}}
\put(7,5){\circle*{.3}}
\put(4,7){\circle*{.3}}
\put(4,1){\line(-3,2)3}
\put(4,1){\line(-1,2)1}
\put(4,1){\line(1,2)1}
\put(4,1){\line(3,2)3}
\put(4,7){\line(-3,-2)3}
\put(4,7){\line(-1,-2)1}
\put(4,7){\line(1,-2)1}
\put(4,7){\line(3,-2)3}
\put(1,3){\line(0,1)2}
\put(1,3){\line(1,1)2}
\put(1,3){\line(2,1)4}
\put(3,3){\line(-1,1)2}
\put(3,3){\line(2,1)4}
\put(5,3){\line(-2,1)4}
\put(5,3){\line(1,1)2}
\put(7,3){\line(-2,1)4}
\put(7,3){\line(-1,1)2}
\put(7,3){\line(0,1)2}
\put(3.85,.3){$0$}
\put(.35,2.85){$a$}
\put(2.35,2.85){$b$}
\put(5.4,2.85){$c$}
\put(7.4,2.85){$d$}
\put(.35,4.85){$d'$}
\put(2.35,4.85){$c'$}
\put(5.4,4.85){$b'$}
\put(7.4,4.85){$a'$}
\put(3.85,7.4){$1$}
\put(3.2,-.75){{\rm Fig.~3}}
\put(-2,-1.55){complemented poset with an antitone involution}
\end{picture}
\end{center}

\vspace*{8mm}

Evidently, $\mathbf P$ is neither a lattice nor distributive since
\[
L\big(U(a,b),c\big)=L(d',1)=L(d')\neq L(0)=LU(0)=LU\big(L(a,c),L(b,c)\big).
\]
We have

Boolean elements: $0$, $a$, $b$, $c$, $d$, $a'$, $b'$, $c'$, $d'$, $1$, \\
maximal ideals: $L(a')$, $L(b')$, $L(c')$, $L(d')$, \\
ultrafilters: $U(a)$, $U(b)$, $U(c)$, $U(d)$, \\
prime ideals: $L(a')$, $L(d')$, \\
prime filters: $U(a)$, $U(b)$, \\
c-ideals: $L(0)$, $L(a)$, $L(b)$, $L(c)$, $L(d)$, $L(a')$, $L(b')$, $L(c')$, $L(d')$, $L(1)$, \\
c-filters: $U(0)$, $U(a)$, $U(b)$, $U(c)$, $U(d)$, $U(a')$, $U(b')$, $U(c')$, $U(d')$, $U(1)$.

The complementation is an antitone involution. Thus the assumption $x\leq x''$ from Corollary~\ref{cor1} is satisfied. If we consider the prime filter $F=U(d)$ and the ideal $I=L(a)$ then $I\cap F=\emptyset$ and there exists a c-ideal $J=I(d')$ with $I\subseteq J$ and $J\cap F=\emptyset$.
\end{example}

Now we can formulate our second Separation Theorem for distributive complemented posets with antitone complementation satisfying the Descending Chain Condition, in particular for finite distributive complemented posets with an antitone involution. Here we need not assume that the complementation satisfies $x\leq x''$ nor that the filter in question satisfies the c-condition.

\begin{theorem}\label{th4}
{\rm(}{\bf Second Separation Theorem}{\rm)} Let $\mathbf P=(P,\leq,{}',0,1)$ be a distributive complemented poset with antitone complementation satisfying the Descending Chain Condition and let $I$ be an ideal and $F$ an ultrafilter of $\mathbf P$. Then there exists some $g\in F$ with $U(g)=F$. Now assume that $x\wedge g$ exists for every $x\in P\setminus F$ and that $I\cap F=\emptyset$. Then there exists some c-ideal $J$ of $\mathbf P$ with $I\subseteq J$ and $J\cap F=\emptyset$.
\end{theorem}

\begin{proof}
Let $a\in P\setminus F$ and $b\in P$. Since $\mathbf P$ satisfies the Descending Chain Condition we have that $F$ is principal according to Lemma~\ref{lem3}, i.e.\ there exists some $g\in F$ with $U(g)=F$. Because of $L(a,g)\subseteq L(a,f)$ for all $f\in F$ we have $UL(a,f)\subseteq UL(a,g)$ for all $f\in F$ and hence $\bigcup\{UL(a,f)\mid f\in F\}=UL(a,g)=UL(a\wedge g)=U(a\wedge g)$ which shows that $\bigcup\{UL(a,f)\mid f\in F\}$ is a filter of $\mathbf P$. According to Theorem~\ref{th5} we conclude that $F$ satisfies the c-condition. Moreover,
\begin{align*}
L(b'') & =L(1,b'')=L\big(U(b',b),b''\big)=LU\big(L(b',b''),L(b,b'')\big)=LU\big(0,L(b,b'')\big)= \\
       & =LUL(b,b'')=L(b,b'')=L(b'',b)=LUL(b'',b)=LU\big(0,L(b'',b)\big)= \\
       & =LU\big(L(b',b),L(b'',b)\big)=L\big(U(b',b''),b\big)=L(1,b)=L(b)
\end{align*}
and hence $b''=b$, i.e.\ the complementation is an involution. Now apply Theorem~\ref{th3}.
\end{proof}

\begin{example}

The complemented poset $\mathbf P$ depicted in Fig.~4

\vspace*{-4mm}

\begin{center}
\setlength{\unitlength}{7mm}
\begin{picture}(8,10)
\put(4,1){\circle*{.3}}
\put(1,3){\circle*{.3}}
\put(3,3){\circle*{.3}}
\put(5,3){\circle*{.3}}
\put(7,3){\circle*{.3}}
\put(1,5){\circle*{.3}}
\put(7,5){\circle*{.3}}
\put(1,7){\circle*{.3}}
\put(3,7){\circle*{.3}}
\put(5,7){\circle*{.3}}
\put(7,7){\circle*{.3}}
\put(4,9){\circle*{.3}}
\put(4,1){\line(-3,2)3}
\put(4,1){\line(-1,2)1}
\put(4,1){\line(1,2)1}
\put(4,1){\line(3,2)3}
\put(4,9){\line(-3,-2)3}
\put(4,9){\line(-1,-2)1}
\put(4,9){\line(1,-2)1}
\put(4,9){\line(3,-2)3}
\put(1,3){\line(0,1)4}
\put(1,3){\line(1,1)4}
\put(3,3){\line(-1,1)2}
\put(3,3){\line(1,1)4}
\put(5,3){\line(-1,1)4}
\put(5,3){\line(1,1)2}
\put(7,3){\line(-1,1)4}
\put(7,3){\line(0,1)4}
\put(1,5){\line(1,1)2}
\put(7,5){\line(-1,1)2}
\put(3.85,.3){$0$}
\put(.35,2.85){$a$}
\put(2.35,2.85){$b$}
\put(5.4,2.85){$c$}
\put(7.4,2.85){$d$}
\put(.35,4.85){$e$}
\put(7.4,4.85){$e'$}
\put(.35,6.85){$d'$}
\put(2.35,6.85){$c'$}
\put(5.4,6.85){$b'$}
\put(7.4,6.85){$a'$}
\put(3.85,9.4){$1$}
\put(3.2,-.75){{\rm Fig.~4}}
\put(-3.5,-1.55){distributive complemented poset with an antitone involution}
\end{picture}
\end{center}

\vspace*{8mm}

is distributive, but not a semilattice. We have

Boolean elements: $0$, $a$, $b$, $c$, $d$, $e$, $a'$, $b'$, $c'$, $d'$, $e'$, $1$, \\
maximal ideals: $L(a')$, $L(b')$, $L(c')$, $L(d')$, \\
ultrafilters: $U(a)$, $U(b)$, $U(c)$, $U(d)$, \\
c-ideals: $L(0)$, $L(a)$, $L(b)$, $L(c)$, $L(d)$, $L(e)$, $L(a')$, $L(b')$, $L(c')$, $L(d')$, $L(e')$, $L(1)$.

One can easily check that for the ultrafilter $F=U(b)$ we have $x\wedge b=0$ for all $x\in P\setminus F$. Thus the assumptions of Theorem~\ref{th4} are satisfied. Now if $I$ denotes the ideal $L(e')$ then $I\cap F=\emptyset$ and there exists a c-ideal $J=L(b')$ with $I\subseteq J$ and $J\cap F=\emptyset$.
\end{example}

Authors' addresses:

Ivan Chajda \\
Palack\'y University Olomouc \\
Faculty of Science \\
Department of Algebra and Geometry \\
17.\ listopadu 12 \\
771 46 Olomouc \\
Czech Republic \\
ivan.chajda@upol.cz

Miroslav Kola\v r\'ik \\
Palack\'y University Olomouc \\
Faculty of Science \\
Department of Computer Science \\
17.\ listopadu 12 \\
771 46 Olomouc \\
Czech Republic \\
miroslav.kolarik@upol.cz

Helmut L\"anger \\
TU Wien \\
Faculty of Mathematics and Geoinformation \\
Institute of Discrete Mathematics and Geometry \\
Wiedner Hauptstra\ss e 8-10 \\
1040 Vienna \\
Austria, and \\
Palack\'y University Olomouc \\
Faculty of Science \\
Department of Algebra and Geometry \\
17.\ listopadu 12 \\
771 46 Olomouc \\
Czech Republic \\
helmut.laenger@tuwien.ac.at
\end{document}